\documentclass[12pt,reqno]{amsart}
\usepackage{a4wide}
\numberwithin{equation}{section}
\usepackage{mathrsfs}
\usepackage{amsfonts}
\usepackage{amsmath}
\usepackage{stmaryrd}
\usepackage{amssymb}
\usepackage{amsthm}
\usepackage{mathrsfs}
\usepackage{url}
\usepackage{amsfonts}
\usepackage{amscd}
\usepackage{indentfirst}
\usepackage{enumerate}
\usepackage{amsmath,amsfonts,amssymb,amsthm}
\usepackage{amsmath,amssymb,amsthm,amscd}
\usepackage{graphicx,mathrsfs}
\usepackage{appendix}

\usepackage[numbers,sort&compress]{natbib}

\usepackage{color}
 \usepackage[colorlinks, linkcolor=blue, citecolor=blue]{hyperref}

\newtheorem{teo}{Theorem}[section]
\newtheorem{lemma}[teo]{Lemma}
\newtheorem{prop}[teo]{Proposition}

\newtheorem{rem}[teo]{Remark}

\numberwithin{equation}{section}

\newcommand\R{\mathbb R}

\newcommand\mbb\mathbb
\newcommand\mbf\mathbf
\newcommand\mcal\mathcal
\newcommand\mfrak\mathfrak
\newcommand\mrm\mathrm
\newcommand\msf\mathsf
\renewcommand\a\alpha
\renewcommand\b\beta
\newcommand\g\gamma
\newcommand\G\Gamma
\renewcommand\d\delta
\newcommand\D\Delta
\newcommand\e\varepsilon
\newcommand\z\zeta
\renewcommand\t\theta
\newcommand\Th\Theta
\newcommand\la\lambda
\newcommand\La\Lambda
\newcommand\s\sigma
\newcommand\si\varsigma
\newcommand\Si\Sigma
\newcommand\ups\upsilon
\newcommand\U\Upsilon
\newcommand\ph\varphi
\renewcommand\o\omega
\renewcommand\O\Omega
\newcommand\wt\widetilde
\newcommand\wh\widehat
\newcommand\ol\overline
\newcommand\ul\underline
\newcommand\mr\mathring
\newcommand\ub\underbrace
\newcommand\pa\partial
\newcommand\n\nabla
\newcommand\fa\forall
\newcommand\ex\exists
\newcommand\es\emptyset
\newcommand\wk\rightharpoonup
\newcommand\inc\hookrightarrow
\newcommand\linf\varliminf
\newcommand\lsup\varlimsup
\newcommand\os\overset
\newcommand\us\underset
\newcommand\sr\stackrel
\newcommand\Ot\Leftarrow
\newcommand\To\Rightarrow
\newcommand\map\mapsto
\newcommand\ot\leftarrow
\newcommand\lot\longleftarrow
\newcommand\lto\longrightarrow
\newcommand\tot\leftrightarrow
\newcommand\ltot\longleftrightarrow
\newcommand\sm\backslash
\renewcommand\Cup\bigcup
\renewcommand\Cap\bigcap
\newcommand\sub\subset
\newcommand\Sub\Subset
\newcommand\sne\subsetneq
\newcommand\bus\supset
\newcommand\Bus\Supset

\newcommand\eq\equiv
\newcommand\ox\otimes
\newcommand\Ox\bigotimes
\newcommand\pl\oplus
\newcommand\Pl\bigoplus
\newcommand\x\times
\renewcommand\c\circ
\newcommand\q\quad
\renewcommand\l\left
\renewcommand\r\right
\newcommand\fr\frac

\allowdisplaybreaks

\begin{document}

\title[Critical points of positive solutions]
{The spectral gap to torsion problem
for some non-convex domains}
\author[H. Chen and P. Luo]{Hua Chen and  Peng Luo}

\address[Hua Chen]{School of Mathematics and Statistics, Wuhan University, Wuhan 430072,  China}
 \email{chenhua@whu.edu.cn}
 \address[Peng Luo]{School of Mathematics and Statistics, Central China Normal University, Wuhan 430079,  China }
 \email{pluo@mail.ccnu.edu.cn}

\begin{abstract}
In this paper we study the following torsion problem
\begin{equation*}
\begin{cases}
-\Delta u=1~&\mbox{in}\ \Omega,\\[1mm]
u=0~&\mbox{on}\ \partial\Omega.
\end{cases}
\end{equation*}
Let $\Omega\subset \R^2$ be a bounded, convex domain and $u_0(x)$ be the solution of above problem with  its maximum $y_0\in \Omega$.  Steinerberger \cite{Ste18} proved that there are universal constants $c_1, c_2>0$ satisfying
\begin{equation*}
\lambda_{\max}\left(D^2u_0(y_0)\right)\leq -c_1\mbox{exp}\left(-c_2\frac{\text{diam}(\Omega)}{\mbox{inrad}(\Omega)}\right).
\end{equation*}
And in \cite{Ste18} he proposed following open problem:

\vskip 0.2cm

\noindent \emph{``Does above result hold true on domains that are not convex but merely simply connected or perhaps only bounded? The proof uses convexity of the domain $\Omega$ in a very essential way and it is not clear to us whether the statement remains valid in other settings."}

\vskip 0.2cm

\noindent Here by some new idea involving the computations on Green's function, we compute
the spectral gap $\lambda_{\max}D^2u_0(y_0)$
for some non-convex smooth bounded domains, which gives a negative answer to above open problem.
\end{abstract}

\date{\today}
\maketitle

\section{Introduction and main results}\label{s0}
\setcounter{equation}{0}

In this paper, we consider  the following   torsion problem
 \begin{equation}\label{1h}
\begin{cases}
-\Delta u=1~&\mbox{in}\ \Omega,\\[1mm]
u=0~&\mbox{on}\ \partial \Omega.
\end{cases}
\end{equation}
Problem \eqref{1h} is a classical topic
in PDEs, with references dating back to St. Venant(1856).
From then, many results are devoted to analysis the
qualitative properties of the positive solutions.
A very interest problem is the location and the number of the critical points of above positive solutions. This is related with the level sets of the positive solutions. For a more general case, the following nonlinear problem
 \begin{equation*}
\begin{cases}
-\Delta u=f(u)~&\mbox{in}\ \Omega,\\[1mm]
u=0~&\mbox{on}\ \partial \Omega.
\end{cases}
\end{equation*}
has also been considered widely. For example, one can refer to \cite{BL1976,GG2019,G2020,K1985,LR2019,ML71} and the related references.

A well-known and seminal result is the fundamental theorem in
Gidas, Ni and Nirenberg \cite{GNN1979} by moving plane.
Gidas-Ni-Nirenberg's Theorem
 shows that the uniqueness of the critical points is related to the shape of the superlevel sets.
Although there are some conjectures on the uniqueness of the critical point in more general convex domains, this seems to be a very difficult problem. And another important result is \cite{CC1998}, which  holds for a wide class of nonlinearities $f$ without the symmetry assumption on $\Omega$ and for semi-stable solutions. For further results, we can refer to \cite{B2018,HNS2016,M2016} and references therein.

When $f(u)\equiv 1$, the torsion function  seems to be the classical object in the study of level sets of elliptic equations. First from \cite{ML71}, we know that the level sets are convex and there is a unique global maximum of the
torsion function on planar convex domains.
And then the eccentricity of the level sets close to the (unique) maximum point $y_0$ is determined by the eigenvalues of the Hessian $D^2u(y_0)$.
Let $\lambda_1$ and $\lambda_2$ are two eigenvalues of $D^2u(y_0)$, then directly, we have
\begin{equation*}
\lambda_1,\lambda_2\leq 0~\mbox{and}~\lambda_1+\lambda_2=\mbox{tr}~D^2u(y_0)=\Delta u(y_0)=-1.
\end{equation*}
This gives us that the level sets will be highly eccentric if one of the two eigenvalues is close to $0$.  In this aspect, Steinerberger \cite{Ste18} gave a beautiful description, which shows that the level sets aren't highly eccentric  for any convex domain $\Omega$ and can be stated as follows.

\vskip 0.2cm

\noindent \textbf{Theorem A.} \emph{Let $\Omega\subset \R^2$ be a bounded,  convex domain and $u_0(x)$ be the solution of problem \eqref{1h} with  its maximum $y_0\in \Omega$. There are universal constants $c_1, c_2>0$ such that}
\begin{equation}\label{5-10-1}
\lambda_{\max}\left(D^2u_0(y_0)\right)\leq -c_1\mbox{exp}\left(-c_2\frac{\text{diam}(\Omega)}{\mbox{inrad}(\Omega)}\right).
\end{equation}
Also Steinerberger \cite{Ste18} gave some details to show that the above result has the sharp scaling. Above Theorem A was proved by Fourier analysis in \cite{Ste18} and highly depends on the convexity of the domain $\Omega$. Next at page 1616 of \cite{Ste18}, Steinerberger proposed the following \textbf{open problem}:

 \vskip 0.2cm

\noindent \textbf{Problem A}. Convexity of the Domain. \emph{Does Theorem A also hold true on domains that are not convex but merely simply connected or perhaps only bounded? The proof uses convexity of the domain $\Omega$ in a very essential way and it is not clear to us whether the statement remains valid in other settings.}

 \vskip 0.2cm

In this paper, we devote to give some answer to above Problem A.
To study Problem A, we will compute the Hessian of the torsion function at the maximum point on a simple non-convex domain.
For example, we suppose that $\Omega_\varepsilon=\O\backslash B(x_0,\varepsilon)$ with $x_0\in\O$ and $B(x_0,\varepsilon)$ denote the ball centered at $x_0$ and radius $\e$, $u_\e$
is the solution of
 \begin{equation}\label{aa2}
\begin{cases}
-\Delta u=1~&\mbox{in}\ \O_\varepsilon,\\[1mm]
u=0~&\mbox{on}\ \partial\O_\varepsilon.
\end{cases}
\end{equation}
And then we have following result.
\begin{teo}\label{th1.1}
Let $\Omega\subset \R^2$ be a bounded and convex domain, $y_0$ is the maximum point  of $u_0(x)$ as in Theorem A.
Suppose that $u_\e(x)$ is the  solution of problem \eqref{aa2}
with its maximum $x_\e\in\Omega_\e$. Let  $\lambda_1$ and $\lambda_2$ be two eigenvalues of $D^2u_0(x)$ at $y_0$, then
\begin{equation*}
 \lim_{\e\to 0}\lambda_{\max}\big(D^2u_\e(x_\e)\big)=
 \begin{cases}
 \max\big\{\lambda_1,\lambda_2\big\} &\mbox{if} ~x_0\neq y_0,\\[2mm]
 \max\big\{\lambda_1,\lambda_2,-|\lambda_2-\lambda_1|\big\} &\mbox{if}~x_0= y_0.
 \end{cases}
\end{equation*}
\end{teo}
\begin{rem}
Taking $\O\subset\R^2$ a bounded and convex domain,
$\Omega_\varepsilon=\O\backslash B(x_0,\varepsilon)$ with $x_0=y_0$ and $\varepsilon$ small, if we suppose that \eqref{5-10-1} is true for $\Omega_\e$, then there exist two positive constants $c_3$ and $c_4$, which is independent with $\e$, such that
\begin{equation}\label{a5-10-1}
\lambda_{\max}\left(D^2u_\e(x_\e)\right)\leq -c_1 \exp\left(-c_2\frac{\text{diam}(\Omega_\e)}{\mbox{inrad}(\Omega_\e)}\right)
\leq -c_3\exp\left(-c_4\frac{\text{diam}(\Omega)}{\mbox{inrad}(\Omega)}\right).
\end{equation}
On the other hand,
moreover if we suppose $\lambda_1=\lambda_2$ (for example $\Omega=B(0,1)$),
then
  Theorem \ref{th1.1} gives us
  \begin{equation*}
 \lim_{\e\to 0}\lambda_{\max}\big(D^2u_\e(x_\e)\big)=0,
\end{equation*}
which is a contradiction with \eqref{a5-10-1}.
Hence we deduce that $\eqref{5-10-1}$ doesn't hold for above non-convex domain $\Omega_\e$, which  gives a negative answer to above Problem A in \cite{Ste18}.  And then in this case, we find that the level sets of the  torsion function  are highly eccentric.
\end{rem}
\begin{rem} Our crucial  ideas are as follows.
To compute  the eigenvalues of the Hessian of $u_\e(x)$ at the maximum point on $\Omega_\e$, a first step is to find the location of the maximum point $x_\e$. And then we need to analyze the asymptotic
behavior of $D^2 u_\e(x_\e)$.
It is well known that  $u_0(x)$ and $u_\e(x)$ are represented by corresponding Green's function. Hence we write  $u_\e(x)$ by the basic Green's function and then analyze the properties of Green's function on $\Omega_\e$.
To be specific, we will establish the basic estimate near $\partial B(x_0,\e)$:
\begin{equation*}
u_\varepsilon(x)=
 u_0(x) + \frac{\log  |x-x_0| }{|\log \e|} \Big(u_0\big(x_0\big)+o(1)\Big)+o\big(1\big).
\end{equation*}
Furthermore, another crucial result is to derive that $u_\e(x)$ and $u_0(x) + \frac{\log  |x-x_0| }{|\log \e|} u_0(x_0)$ are close in the $C^2$-topology  in $B(x_0,d)\setminus B(x_0,\e)$ for some small fixed $d>0$, which can be found in Proposition \ref{alemma-2} below.
\end{rem}

\begin{rem}
Now we would like to point out that $\Omega_\e$ in Theorem \ref{th1.1} can be replaced by $$\Omega'_\e=\Omega\backslash A_\e~\mbox{with}~A_\e=\e \big(A-x_0\big)+x_0~\mbox{and}~x_0\in A\cap \Omega,$$
where $A$ is a convex domain in $\R^2$ and $A-x_0=\{x, x+x_0\in A\}$. Since this is not essential, we omit the details.
\end{rem}

\begin{rem}
We point out one possible application in
the study of Brownian motion, which is also stated in \cite{Ste18}: we recall that the torsion function $u(x)$ also describes the expected lifetime of Brownian motion $\omega_x(t)$ started in $x$ until it first touches the boundary. If one moves away from the point in which lifetime is maximized, then the expected lifetime in a neighborhood is determined by the eccentricity of the level set.
\end{rem}

The paper is organized as follows.  In Section \ref{s1}, we recall some properties of the Green's function and split our solution $u_\e$ in different parts which will be estimated
in the next section. In Section \ref{s3}, we compute the terms $u_\e$, $\nabla u_\e$ and $\nabla^2 u_\e$.
Section \ref{s9} is devoted to the proof of Theorem \ref{th1.1}.

\section{Properties of the Green's function and splitting of the solution $u_\varepsilon$}\label{s1}

First we recall that, for $(x,y)\in\O\times\O$, $x\ne y$, the Green's function $G(x,y)$ verifies
\begin{equation*}
\begin{cases}
-\D_x G(x,y)=\delta(y)&\hbox{in }\O,\\[1mm]
G(x,y)=0&\hbox{on }\partial\O,
\end{cases}
\end{equation*}
in the sense of distribution. Next we recall the classical representation formula,
\begin{equation}\label{G}
G(x,y)=-\frac{1}{2\pi}\log \big|x-y\big|-H(x,y),
\end{equation}
where  $H(x,y)$ is the {\em regular part of the Green's function}.
Since in the paper we need to consider the Green's function in different domains, we would like to denote by $G_U(x,y)$  as  the Green's function on $U$. And we have following facts on the harmonic function which can be found in \cite{GT1983}.
\begin{lemma}\label{lem2-3}
Let $u(x)$ be a $harmonic$ function in $U\subset\subset \R^2$, then
\begin{equation}\label{a11-02-01}
\big|\nabla u(x)\big|\leq \frac{2}{r}\sup_{\partial B(x,r)} |u|,~\, \mbox{for}\,~B(x,r)\subset\subset U.
\end{equation}
\end{lemma}
\begin{lemma}[Green's representation formula]\label{lem2-a}
If $u\in C^2(\overline{U})$, then it holds
\begin{equation}\label{grf}
u(x)=-\int_{\partial U}u(y)\frac{\partial G_U(x,y)}{\partial \nu_y}d\sigma(y)
-\int_{U} \Delta u(y) G_U(x,y)dy,~\, \mbox{for}~x\in U,
\end{equation}
where $\nu_y$ is the outer normal vector on $\partial U$.
\end{lemma}

Let us denote by $G_0(w,s)$ the {\em Green's function} of $\R^2\backslash B(0,1)$ given by (see \cite{BF1996})
\begin{equation*}
G_0(w,s)=
-\frac{1}{2\pi}\left(\log \big|{w-s}\big|-\log \big||w|s-\frac{w}{|w|}\big|\right).
\end{equation*}
By a straightforward computation, we have
 \begin{equation}\label{daaa11-02-10}
\frac{\partial G_0(w,s)}{\partial \nu_s}=
 \frac{1-|w|^2}{2\pi|w-s|^2},~\mbox{for}~|w|>1,~|s|=1~\mbox{and}~\nu_s=-s.
\end{equation}
\begin{rem}
Let us point out that  the {\em Green's function} $G_0(w,s)$ of $\R^2\backslash B(0,1)$  and
the {\em Poisson kernel} of $B(0,1)$  has the same formula (see \cite{BF1996}). This will be used to compute some  integral in $\R^2\backslash B(0,1)$.
\end{rem}
Next lemma will be basic and useful in the following computations in next section.
 \begin{lemma}\label{G3}
Let $v_\e(x)$ be the function which verifies
 \begin{equation*}
\begin{cases}
\Delta v_\e(x)=0&~\mbox{in}~\O\setminus B(x_0,\e),\\[1mm]
v_\e(x)=0&~\mbox{on}~\partial\O,\\[1mm]
v_\e(x)=1&~\mbox{on}~\partial B(x_0,\e).
\end{cases}
\end{equation*}
Then we have that
\begin{equation*}
v_\e(x)=-\frac{2\pi}{\log\e}\Big( 1-\frac{2\pi  H(x_0,x_0)}{\log\e}\Big)G(x,x_0)+O\left(\frac1{|\log\e|^2}\right).
\end{equation*}
\end{lemma}
\begin{proof}
First we define
$$w_\e(x)=\frac1{H(x_0,x_0)}\left[\frac{\log\e}{2\pi} v_\e(x)+G(x,x_0)\right].$$
Then it holds
 \begin{equation*}
\begin{cases}
\Delta  w_\e(x)=0&~\mbox{in}~\O\setminus B(x_0,\e),\\[1mm]
w_\e(x)=0&~\mbox{on}~\partial\O,\\[1mm]
w_\e(x)= \frac1{H(x_0,x_0)}\Big(\frac{\log\e}{2\pi}+G(x,x_0)\Big)=-1+O(\e)&~\mbox{on}~\partial B(x_0,\e).
\end{cases}
\end{equation*}
Hence repeating above procedure, we can find
 \begin{equation}\label{ads}
\begin{cases}
\Delta \left[\frac{\log\e}{2\pi}w_\e(x)-G(x,x_0)\right]=0&~\mbox{in}~\O\setminus B(x_0,\e),\\[1.5mm]
\frac{\log\e}{2\pi}w_\e(x)-G(x,x_0)=0&~\mbox{on}~\partial\O,\\[1.5mm]
\frac{\log\e}{2\pi}w_\e(x)-G(x,x_0)= H(x,x_0)+O(\e|\log\e|)&~\mbox{on}~\partial B(x_0,\e).
\end{cases}
\end{equation}
  Then
by the maximum principle and \eqref{ads}, we get that
$$\frac{\log\e}{2\pi}w_\e(x)-G(x,x_0)=O(1)~\,~\mbox{in}~\O\setminus B(x_0,\e),$$
which gives
\begin{equation*}
\begin{split}
w_\e(x)=&\frac{2\pi}{\log\e}G(x,x_0)+O\left(\frac1{|\log\e|}\right).
\end{split}
\end{equation*}
Hence coming back to $v_\e(x)$, we find
\begin{equation*}
\begin{split}
v_\e(x)=&\frac{2\pi}{\log\e}\Big( H(x_0,x_0)w_\e(x)-G(x,x_0)\Big)\\=&
-\frac{2\pi}{\log\e}\Big( 1-\frac{2\pi  H(x_0,x_0)}{\log\e}\Big)G(x,x_0)+O\left(\frac1{|\log\e|^2}\right),
\end{split}
\end{equation*}
which gives the claim.
\end{proof}

Let $u_0$ and $u_\e$ be solutions of \eqref{1h} and \eqref{aa2} respectively, then we can write down the equation satisfied by $u_\e-u_0$ as follows
\begin{equation}\label{A1}
\begin{cases}
-\Delta \big(u_\varepsilon-u_0\big)=0~&\mbox{in}~\Omega_\varepsilon,\\[1mm]
u_\varepsilon-u_0=0~&\mbox{on}~\partial\Omega,\\[1mm]
u_\varepsilon-u_0=-u_0~&\mbox{on}~\partial B(x_0,\e).
\end{cases}
\end{equation}
Now by Green's representation formula \eqref{grf},  we get
\begin{equation}\label{6-25-51}
u_\varepsilon(x)=u_0(x)+
\int_{\partial B(x_0,\e)} \frac{\partial G_\varepsilon(x,z)}{\partial\nu_z}u_0(z)d\sigma(z),
\end{equation}
where $\nu_z=-\frac{z-x_0}{|z-x_0|}$ is the outer normal vector of $\partial \big(\R^2\backslash  B(x_0,\varepsilon)\big)$ and $G_\varepsilon(x,z)$ is the Green's function of $-\Delta$ in $\O_\varepsilon$ with zero Dirichlet boundary condition.
Now we set
$$x=x_0+\e w,\ z=x_0+\e s\,\,~\mbox{and}~
F_\e(w,s)= G_\varepsilon(x_0+\e w,x_0+\e s),$$
then  \eqref{6-25-51} becomes
\begin{equation}\label{6-26-1}
 u_\varepsilon(x)=u_0(x)+K_{\e}(w)+L_{\e}(w),
\end{equation}
where \begin{equation*}
K_{\e}(w):=\int_{\partial B(0,1)} \frac{\partial G_0(w,s)}{\partial\nu_s}
u_0(x_0+\e s)
d\sigma(s),
\end{equation*}
and
\begin{equation*}
L_{\e}(w):=\int_{\partial B(0,1)}\left(\frac{\partial F_\varepsilon(w,s)}{\partial\nu_s}-\frac{\partial G_0(w,s)}{\partial\nu_s}\right)u_0(x_0+\e s)d\sigma(s),
\end{equation*}
with $\nu_s=-\frac{s}{|s|}$  the outer normal vector of $\partial \big(\R^2\backslash  B(0,1)\big)$.

\section{Asymptotic analysis on $u_\e$}\label{s3}

To compute $\lambda_{\max}\big(D^2u_\e(x)\big)$ at the maximum point of $u_\e$, the first thing is to find the location of the maximum point of $u_\e$. Here we divide $\Omega_\e$ into the following two cases:
\vskip 0.2cm
\begin{description}
  \item[(1)] $x$ is far away from $x_0$, namely $|x-x_0|\ge C>0$.
  \vskip 0.2cm
  \item[(2)] $x$ is near $x_0$, namely $|x-x_0|=o(1)$.
\end{description}
\vskip 0.2cm
And we will find that the behavior of $u_\e(x)$ near $\partial B(x_0,\varepsilon)$ is crucial and a key point is to understand the limit of $G_\varepsilon(x,y)$ according to the location of $x$.

\begin{lemma}\label{alemma-1}
Let  $u_0$ and $u_\e$ be solutions of \eqref{1h} and \eqref{aa2} respectively. Then for  any fixed $r>0$, it holds
\begin{equation*} 
 u_\varepsilon(x)\rightarrow  u_0(x)~\mbox{uniformly in}~C^2\big(\Omega\backslash B(x_0,r)\big).
\end{equation*}
\end{lemma}
\begin{proof}
First for any $x\in \Omega\backslash B(x_0,r)$, we
know
\begin{equation}\label{Asfs}
\begin{split}
u_\varepsilon(x)=&u_0(x)+
\int_{\partial B(x_0,\e)} \frac{\partial G_\varepsilon(x,z)}{\partial\nu_z}u_0(z)d\sigma(z)\\=&
u_0(x)+
\int_{\partial B(x_0,\e)} \frac{\partial G_\varepsilon(x,z)}{\partial\nu_z}\Big(u_0(x_0)+O\big(\e\big)\Big)d\sigma(z)
\\=&
u_0(x)+u_0(x_0)
\int_{\partial B(x_0,\e)} \frac{\partial G_\varepsilon(x,z)}{\partial\nu_z} d\sigma(z)
+O\big( \e\big)
\int_{\partial B(x_0,\e)}\Big| \frac{\partial G_\varepsilon(x,z)}{\partial\nu_z}\Big|d\sigma(z).
\end{split}
\end{equation}
Now let $x=x_0+\e w$ and $z=x_0+\e s$, then using \eqref{daaa11-02-10}, we have
\begin{equation}\label{a6-25-2}
\begin{split}
\frac{\partial G_\varepsilon(x,z)}{\partial\nu_z}
=&\frac{1}{\e}\left(\frac{\partial G_0(w,s)}{\partial\nu_s}+\Big(\frac{\partial F_\varepsilon(w,s)}{\partial\nu_s}-\frac{\partial G_0(w,s)}{\partial\nu_s}\Big)\right)\\=&
O\left( \frac{1}{\e}\Big(1+\big|\frac{\partial F_\varepsilon(w,s)}{\partial\nu_s}-\frac{\partial G_0(w,s)}{\partial\nu_s}\big|\Big)\right).
\end{split}\end{equation}

On the other hand, we can verify that
\begin{equation}\label{adsd}
\begin{cases}
\Delta_w \Big(\frac{\partial F_\varepsilon(w,s)}{\partial\nu_s}-\frac{\partial G_0(w,s)}{\partial\nu_s}\Big)=0&~\mbox{in}~\frac{\O-x_0}\e\setminus B(0,1),\\[2mm]
\Big(\frac{\partial F_\varepsilon(w,s)}{\partial\nu_s}-\frac{\partial G_0(w,s)}{\partial\nu_s}\Big)=0&~\mbox{on}~\partial B(0,1),\\[2mm]
\Big(\frac{\partial F_\varepsilon(w,s)}{\partial\nu_s}-\frac{\partial G_0(w,s)}{\partial\nu_s}\Big)= \frac{|w|^2-1}{2\pi |w-s|^2}=O\big(1\big) &~\mbox{on}~\frac{\partial\O-x_0}\e.
\end{cases}
\end{equation}
By the maximum principle and \eqref{adsd}, we get that
\begin{equation}\label{adsda}\frac{\partial F_\varepsilon(w,s)}{\partial\nu_s}-\frac{\partial G_0(w,s)}{\partial\nu_s}=O(1)~\,~\mbox{in}~\frac{\O-x_0}\e\setminus B(0,1).
\end{equation}
Hence from \eqref{a6-25-2} and \eqref{adsda}, we find
\begin{equation}\label{6-25-2}
\begin{split}
\frac{\partial G_\varepsilon(x,z)}{\partial\nu_z}
=O\left( \frac{1}{\e}\right).
\end{split}\end{equation}

Also defining $ v(x):= \displaystyle\int_{\partial B(x_0,\e)}\frac{\partial G_\varepsilon(x,z)}{\partial\nu_z}d\sigma(z)$, then it holds
  \begin{equation*}
\begin{cases}
\Delta_xv(x)=0&~\mbox{in}~\O\setminus B(x_0,\e),\\[1mm]
v(x)=0&~\mbox{on}~\partial\O,\\[1mm]
v(x)=-1&~\mbox{on}~\partial B(x_0,\e).
\end{cases}
\end{equation*}
Then using Lemma \ref{G3}, we have
   \begin{equation} \label{6-25-8}
   v(x)=
 \frac{2\pi}{\log\e}\Big( 1-\frac{2\pi H(x_0,x_0)}{\log\e}\Big)G(x,x_0)+O\left(\frac1{|\log\e|^2}\right).
 \end{equation}
 Hence from \eqref{Asfs}, \eqref{6-25-2}  and \eqref{6-25-8}, we find
 \begin{equation*}
 \begin{split}
u_\varepsilon(x)=&  
u_0(x)+u_0(x_0) \left(
 \frac{2\pi}{\log\e}\Big( 1-\frac{2\pi H(x_0,x_0)}{\log\e}\Big)G(x,x_0)+O\Big(\frac1{|\log\e|^2}\Big) \right)
+O\big( \e\big) \\=& 
u_0(x)+
O\left(\frac1{|\log\e|}\right)\,~
\mbox{uniformly in}~C\big(\Omega\backslash B(x_0,r)\big).
\end{split}
\end{equation*}

On the other hand,  for any fixed $x\in \Omega\backslash B(x_0,r)$, by \eqref{G},  we can verify that
\begin{equation}\label{6-25-3}
G(x,z),~\big|\frac{\partial G(x,z)}{\partial\nu_z}\big|,~|\nabla_xG(x,z)|~\mbox{and}~\big|\nabla_x\frac{\partial G(x,z)}{\partial\nu_z}\big|~\mbox{are bounded for}~z\in \partial B(x_0,\e).
\end{equation}
 And
by \eqref{A1}, it follows
\begin{equation*}
-\Delta_x \Big( \frac{\partial G_\varepsilon(x,z)}{\partial\nu_z}- \frac{\partial G(x,z)}{\partial\nu_z}\Big)=0~ ~\mbox{in}~\Omega_\varepsilon.
\end{equation*}
Since $B\big(x,\frac{r}{2}\big) \subset\subset \Omega_\e$,
 using Lemma \ref{lem2-3}, \eqref{6-25-2} and \eqref{6-25-3}, we get,  for $x\in \Omega\backslash B(x_0,r)$ and $z\in \partial B(x_0,\e)$,
\begin{equation}\label{6-25-4}
\begin{split}
\left|\nabla_x\Big( \frac{\partial G_\varepsilon(x,z)}{\partial\nu_z}- \frac{\partial G(x,z)}{\partial\nu_z}\Big)
\right|=O \left( \Big|\frac{\partial G_\varepsilon(x,z)}{\partial\nu_z}- \frac{\partial G(x,z)}{\partial\nu_z}\Big|\right)=O\left( \frac{1}{\e}\right),
\end{split}\end{equation}
and
\begin{equation}\label{6-25-5}
\begin{split}
 \left|\nabla^2_x\Big( \frac{\partial G_\varepsilon(x,z)}{\partial\nu_z}- \frac{\partial G(x,z)}{\partial\nu_z}\Big)
\right|=O \left( \Big|\frac{\partial G_\varepsilon(x,z)}{\partial\nu_z}- \frac{\partial G(x,z)}{\partial\nu_z}\Big|\right)=O\left( \frac{1}{\e}\right).
\end{split}\end{equation}
Then \eqref{6-25-3}, \eqref{6-25-4} and \eqref{6-25-5} give us that
\begin{equation*}
\begin{split}
\left|\nabla_x \frac{\partial G_\varepsilon(x,z)}{\partial\nu_z}
\right|= O\left( \frac{1}{\e}\right)~\mbox{and}~\left|\nabla^2_x \frac{\partial G_\varepsilon(x,z)}{\partial\nu_z}
\right|= O\left( \frac{1}{\e}\right).
\end{split}\end{equation*}
Hence from above estimates, it follows
  \begin{equation*}
u_\varepsilon(x)=u_0(x)+
O\left(\frac1{|\log\e|}\right)\,~
\mbox{uniformly in}~C^2\big(\Omega\backslash B(x_0,r)\big).
\end{equation*}
\end{proof}
In the rest of this section, we devote to analyze the asymptotic behavior on $u_\e$, $\nabla u_\e$ and
$\nabla^2 u_\e$ near $\partial B(x_0,\varepsilon)$. And using \eqref{6-26-1}, we need to compute the terms
$K_\e$, $\nabla_wK_\e$, $\nabla_w^2K_\e$, $L_\e$, $\nabla_wL_\e$ and $\nabla_w^2L_\e$.
\begin{lemma}\label{lll1}
Let $w=\frac{x-x_0}{\e}$ and if $|x-x_0|\to 0$,
then it holds
\begin{equation}\label{aaas11-15-04}
\begin{split}
 K_{\e}(w)=& - u_0\Big(x_0+ \frac{\varepsilon w}{|w|^2}\Big)+\frac{\varepsilon^2}{2}\Big(1-\frac{1}{|w|^2}\Big),
 \end{split}
\end{equation}
\begin{equation}\label{aaas11-15-04-1}
\begin{split}
\frac{\partial K_{\e}(w)}{\partial w_i}= O\Big(\frac{\varepsilon}{|w|^2}\Big)~\,\mbox{and}~\,
 \frac{\partial^2 K_{\e}(w)}{\partial w_i\partial w_j}= O\Big(\frac{\varepsilon}{|w|^3}\Big).
\end{split}
\end{equation}
\end{lemma}
\begin{proof}
First  taking $\tau=\frac{w}{|w|^2}=\frac{\varepsilon(x-x_0)}{|x-x_0|^2}$ and using \eqref{daaa11-02-10}, we get
\begin{equation*}
\begin{split}
\int_{\partial B(0,1)} &\frac{\partial G_0(w,s)}{\partial\nu_s}u_0(x_0+\e s)d\sigma(s)=-\frac{1}{2\pi }\int_{\partial B(0,1)} \frac{1-|\tau|^2}{|\tau-s|^2}u_0(x_0+\e s)d\sigma(s).
\end{split}
\end{equation*}
Lemma \ref{lem2-a} and \eqref{daaa11-02-10} give us  that for any $\phi\in C^2\big(\overline{B(0,1)}\big)$, it holds
\begin{equation}\label{G2}
\phi(s)=\frac{1}{2\pi}\int_{\partial B(0,1)} \frac{1-|s|^2}{|s-y|^2}\phi(y)d\sigma(y)
-\int_{B(0,1)} \Delta \phi(y) G_0(s,y)dy.
\end{equation}
Hence for $|\tau|<1$ and  choosing $\phi(\tau)=u_0(x_0+\e \tau)$ in \eqref{G2} we find
\begin{equation*}
u_0(x_0+\e \tau)= \frac{1}{2\pi }\int_{\partial B(0,1)} \frac{1-|\tau|^2}{|\tau-s|^2}u_0(x_0+\e s)d\sigma(s)+\frac{\varepsilon^2}{2}\big(1-|\tau|^2\big).
\end{equation*}
From the above computations  we  get
\begin{equation}\label{Bd22}
\begin{split}
K_{\e}(w)=&\int_{\partial B(0,1)}\frac{\partial G_0(w,s)}{\partial\nu_s}u_0(x_0+\e s)d\sigma(s)\\=&- u_0\Big(x_0+ \frac{\varepsilon w}{|w|^2}\Big)+\frac{\varepsilon^2}{2}\Big(1-\frac{1}{|w|^2}\Big).
\end{split}
\end{equation}
And then by differentiating \eqref{Bd22} with respect to $w_i$, we have
\begin{equation}\label{aaab1-15-04}
\begin{split}
\frac{\partial K_{\e}(w)}{\partial w_i}=&
-\frac{\e}{|w|^{2}}\left(\frac{\partial u_0(x_0+ \frac{\varepsilon w}{|w|^2})}{\partial x_i}-2\frac{w_i}{ |w|^2}\sum_{j=1}^2\frac{\partial u_0(x_0+ \frac{\varepsilon w}{|w|^2})}{\partial x_j}w_j\right)
+ \frac{\varepsilon^2 w_i}{|w|^4}.
\end{split}
\end{equation}
Next differentiating \eqref{aaab1-15-04} with respect to $w_i$, we find
\begin{equation}\label{1-15-04a}
\begin{split}
\frac{\partial^2 K_{\e}(w)}{\partial w_i\partial w_j}=&
O\left(\frac{\e}{|w|^{3}}\right).
\end{split}
\end{equation}
Hence \eqref{aaas11-15-04} and \eqref{aaas11-15-04-1} follow by \eqref{Bd22}, \eqref{aaab1-15-04} and \eqref{1-15-04a}.
\end{proof}
\begin{lemma}
Let $w=\frac{x-x_0}{\e}$ and if $|x-x_0|\to 0$,
then it holds
\begin{equation}\label{as11-15-06}
\begin{split}
L_{\e}(w)=&\frac{\log |w| }{|\log \e|} \Big(u_0(x_0)+o(1)\Big).
\end{split}
\end{equation}
\end{lemma}
\begin{proof}
First we define
$$M_{\e,2}(w,s)=\sum_{i=1}^2\left(\frac{\partial G_0(w,s)}{\partial s_i}-\frac{\partial F_\varepsilon(w,s)}{\partial s_i}\right)s_i-\frac{1}{2\pi}  M_{\e,1}(w,s),$$
with
\begin{equation*}
\begin{cases}
\Delta_w {M}_{\e,1}(w,s)=0&~\mbox{in}~\frac{\O-x_0}\e\setminus B(0,1),\\[1mm]
 {M}_{\e,1}(w,s)=0&~\mbox{on}~\partial B(0,1),\\[1mm]
 {M}_{\e,1}(w,s)=1&~\mbox{on}~\frac{\partial\O-x_0}\e.
\end{cases}
\end{equation*}
Then $L_{\e}(w)$ can be written as
\begin{equation}\label{B5}
\begin{split}
L_{\e}(w)=&\underbrace{\frac{1}{2\pi}\int_{\partial B(0,1)} {M}_{\e,1}(w,s)u_0(x_0+\e s)d\sigma(s)}_{:= {L}_{\e,1}(w)} +\underbrace{\int_{\partial B(0,1)} M_{\e,2}(w,s)u_0(x_0+\e s)d\sigma(s)}_{:=L_{\e,2}(w)}.
\end{split}
\end{equation}
Also for any $w\in \frac{\partial\O-x_0}\e$ and $s\in \partial B(0,1)$, it holds
\begin{equation*}
\sum_{i=1}^2\left(\frac{\partial G_0(w,s)}{\partial s_i}-\frac{\partial F_\varepsilon(w,s)}{\partial s_i}\right)s_i= \sum_{i=1}^2 \frac{\partial G_0(w,s)}{\partial s_i} s_i=
\frac{1}{2\pi} \frac{|w|^2-1}{|w-s|^2}.
\end{equation*}
Hence we can verify
\begin{equation}\label{cl1}
\begin{cases}
\Delta_wM_{\e,2}(w,s)=0&~\mbox{in}~\frac{\O-x_0}\e\setminus B(0,1),\\[1mm]
M_{\e,2}(w,s)=0&~\mbox{on}~\partial B(0,1),\\[1mm]
M_{\e,2}(w,s)=\frac{1}{2\pi}\left(\frac{|w|^2-1}{|w-s|^2}-1\right)&~\mbox{on}~\frac{\partial\O-x_0}\e.
\end{cases}
\end{equation}
Since for any $w\in \frac{\partial\O-x_0}\e$  and $s\in \partial B(0,1)$,  we get that
\begin{equation}\label{cl2}
\begin{split}
 \frac{|w|^2-1}{|w-s|^2}-1=& O\Big(\frac{1}{|w|}\Big)=O\big(\e\big).
\end{split}
\end{equation}
Then  by the maximum principle, \eqref{cl1} and \eqref{cl2}, we find
\begin{equation}\label{cl3}
\big|M_{\e,2}(w,s)\big|=O(\varepsilon)~\mbox{for}~w\in\frac{\O-x_0}\e \setminus B(0,1)~\mbox{and}~s\in \partial B(0,1).
\end{equation}
Hence it follows
 \begin{equation}\label{m4}
L_{\e,2}(w)=O\big(\varepsilon\big)\,~\mbox{for}\,~w\in\frac{\O-x_0}\e \setminus B(0,1).
\end{equation}

Next we estimate ${M}_{\e,1}(w,s)$. To do this let us introduce the function $\psi_\e(x,s)$ as follows
$$\psi_\e(x,s):=1- M_{\e,1}\left(\frac{x-x_0}\e,s\right)~\,\mbox{for}~~x\in \O\setminus B(x_0,\e)~\,\mbox{and}\, s\in \partial B(0,1).$$
Then it follows
\begin{equation*}
\begin{cases}
\Delta_x \psi_\e(x,s)=0&~\mbox{in}~\O\setminus B(x_0,\e),\\[1mm]
\psi_\e(x,s)=0&~\mbox{on}~ \partial\O,\\[1mm]
\psi_\e(x,s)=1&~\mbox{on}~\partial B(x_0,\e).
\end{cases}
\end{equation*}
Hence using Lemma \ref{G3} we have that
\begin{equation*}
\psi_\e(x,s)=-\frac{2\pi}{\log\e}\Big( 1-\frac{2\pi H(x_0,x_0)}{\log\e}\Big)G(x,x_0)+O\left(\frac1{|\log\e|^2}\right).
\end{equation*}
 Coming back to ${M}_{\e,1}(w,s)$, we get
 \begin{equation*}
\begin{split}
{M}_{\e,1}(w,s)=&1+\frac{2\pi}{\log\e}\Big( 1-\frac{2\pi  H(x_0,x_0)}{\log\e}\Big)G(\varepsilon w+x_0,x_0)+O\left(\frac1{|\log\e|^2}\right)\\=&
\frac{\log |w|}{|\log\e|}\Big(1+o(1)\Big)
+o\left(\frac{1}{|\log\e|}\right).
\end{split}
\end{equation*}
 In last equality  we used that $\varepsilon |w|=|x-x_0|\rightarrow 0$. And then
 \begin{equation}\label{m3}
\begin{split}
L_{\e,1}(w)= \frac{\log |w|}{|\log\e|}\Big(u_0(x_0)+o(1)\Big)
+o\left(\frac{1}{|\log\e|}\right).
\end{split}
\end{equation}
Then \eqref{as11-15-06} follows by \eqref{B5}, \eqref{m4} and \eqref{m3}.
\end{proof}
Then we have following estimate on $u_\e$  near $\partial B(x_0,\varepsilon)$.
\begin{prop}\label{alemma-2d}
Let $w=\frac{x-x_0}{\e}$ and $|x-x_0|\to 0$,   then it holds
\begin{equation}\label{s11-15-0d1}
 u_\e(x) =
 u_0(x)   + \frac{\log  |x-x_0| }{|\log \e|} \Big(u_0\big(x_0\big)+o(1)\Big)+o(1).
\end{equation}
\end{prop}
\begin{proof}
From \eqref{6-26-1}, \eqref{aaas11-15-04} and \eqref{as11-15-06}, we have
\begin{equation*}
\begin{split}
 u_\e(x) =&
 u_0(x) - u_0\Big(x_0+ \frac{\varepsilon w}{|w|^2}\Big)+\frac{\varepsilon^2}{2}\big(1-\frac{1}{|w|^2}\big)
 +\frac{\log |w| }{|\log \e|} \Big(u_0\big(x_0\big)+o(1)\Big)\\=&
 u_0(x)   + \frac{\log  |x-x_0| }{|\log \e|} \Big(u_0\big(x_0\big)+o(1)\Big)+o(1).
\end{split}\end{equation*}

\end{proof}
Now we continue to compute $\nabla_wL_\e$ and $\nabla_w^2L_\e$.
\begin{lemma}\label{aB3}
Let $w=\frac{x-x_0}{\e}$ and $|x-x_0|\to 0$, we have following results:
 \vskip 0.2cm

 \textup{(1)} For any fixed $C_0>1$, if $|w|\geq C_0$, then it holds
\begin{equation}\label{aast11-23-05}
\begin{split}
\frac{\partial L_{\e}(w)}{\partial w_i}=&
\frac{u_0(x_0)w_i}{|\log\e|\cdot |w|^2}+o\left(\frac{1}{ |w|\cdot |\log\e|}\right),
\end{split}
\end{equation}
and
\begin{equation}\label{aamt11-23-05}
\begin{split}
\frac{\partial^2 L_{\e}(w)}{\partial w_i\partial w_j}=&
\frac{u_0(x_0)}{ |\log \varepsilon|\cdot |w|^2} \Big(\delta_{ij}-\frac{2w_iw_j}{|w|^2}\Big)
+o\left(\frac{1}{ |w|^2\cdot |\log\e|}\right).
\end{split}
\end{equation}

 \textup{(2)}
If $\displaystyle\lim_{\e\to 0}|w|=1$, then it holds
\begin{equation}\label{d11-23-05}
\begin{split}
\Big\langle \nabla_w L_{\e}(w),w\Big\rangle
 \geq \frac{u_0(x_0)}{2|\log \e|}.
\end{split}\end{equation}
\end{lemma}
\begin{proof}[\underline{\textbf{Proof of \eqref{aast11-23-05} and \eqref{aamt11-23-05}}}]
First for $w\in\frac{\O-x_0}\e \setminus B(0,1)~\mbox{and}~s\in \partial B(0,1)$, we know $$B\big(w,\frac{|w|-1}{2}\big)\subset\subset \frac{\O-x_0}\e \setminus B(0,1).$$
Then using
\eqref{a11-02-01}, \eqref{cl1} and \eqref{cl3}, we have
\begin{equation}\label{acl3}
\big|\nabla_wM_{\e,2}(w,s)\big|=O\Big(\frac{\varepsilon}{|w|-1}\Big)
~~\mbox{and}~~
\big|\nabla^2_wM_{\e,2}(w,s)\big|=O\Big(\frac{\varepsilon}{(|w|-1)^2}\Big).
\end{equation}
Hence if $\e|w|\to 0$ and  $|w|\geq C_0$ for any fixed $C_0>1$, then \eqref{acl3} implies
\begin{equation*}
\big|\nabla_wM_{\e,2}(w,s)\big|=O\Big(\frac{\varepsilon}{|w|}\Big)
~~\mbox{and}~~
\big|\nabla^2_wM_{\e,2}(w,s)\big|=O\Big(\frac{\varepsilon}{|w|^2}\Big).
\end{equation*}
And then it follows
\begin{equation}\label{11-21-33}
\frac{\partial {L}_{\e,2}(w)}{\partial w_i}=O\left(\frac{\e}{|w|}\right)~\mbox{and}~
\frac{\partial^2{L}_{\e,2}(w)}{\partial w_i\partial w_j}=O\left(\frac\e{|w|^2}\right).
\end{equation}
 Also $M_{\e,1}(w,s)-
\frac{\log |w|}{|\log\e|}$ is a harmonic function with respect to $w$ in $\frac{\O-x_0}\e\setminus B(0,1)$. Hence
if $\e|w|\to 0$ and  $|w|\geq C_0$ for any fixed $C_0>1$,
using \eqref{a11-02-01} we have
\begin{equation}\label{11-24-18}
\left|\nabla_w\left(M_{\e,1}(w,s)-
\frac{\log |w|}{|\log\e|}\right)\right|=o\left(\frac{1}{ |w|\cdot |\log\e|}\right),
\end{equation}
and
\begin{equation}\label{a11-24-18}
\left|\nabla^2_w\left(M_{\e,1}(w,s)-
\frac{\log |w|}{|\log\e|}\right)\right|=o\left(\frac{1}{ |w|^2\cdot|\log\e|}\right).
\end{equation}
Hence from \eqref{11-24-18} and \eqref{a11-24-18}, we have
\begin{equation}\label{11-21-35}
\frac{\partial L_{\e,1}(w)}{\partial w_i}
=\frac{w_i}{ |\log \varepsilon|\cdot |w|^2}\big(u_0(x_0)+o(1)\big),
\end{equation}
and
\begin{equation}\label{11-21-35a}
\frac{\partial^2 L_{\e,1}(w)}{\partial w_i\partial w_j}
=\frac{1}{ |\log \varepsilon|\cdot |w|^2}\big(u_0(x_0)+o(1)\big)\Big(\delta_{ij}-\frac{2w_iw_j}{|w|^2}\Big).
\end{equation}
And then \eqref{aast11-23-05} and  \eqref{aamt11-23-05} follows by \eqref{B5},
\eqref{11-21-33}, \eqref{11-21-35} and \eqref{11-21-35a}.
\end{proof}
\begin{proof}[\underline{\textbf{Proof of \eqref{d11-23-05}}}]
To consider the case $\displaystyle\lim_{\e\to 0}|w|=1$, we define a new function
$$\eta_{\e,1}(w,s)=\sum_{i=1}^2\left(\frac{\partial G_0(w,s)}{\partial s_i}-\frac{\partial F_\varepsilon(w,s)}{\partial s_i}\right)s_i+\frac{\log|w|}{2\pi\log (r\e)}.$$
Then we can write
\begin{equation*}
\frac{\partial L_{\e}(w)}{\partial w_i}= \int_{\partial B(0,1)}\left(
\frac{\partial\eta_{\e,1}(w,s)}{\partial w_i}-\frac{w_i}{2\pi|w|^2\log (r\e)}
\right)u_0(x_0+\e s)d\sigma(s).
\end{equation*}
Now we can verify
\begin{equation}\label{6-23-15}
\begin{cases}
\Delta_w\eta_{\e,1}(w,s)=0&~\mbox{in}~\frac{\O-x_0}\e\setminus B(0,1),\\[1mm]
\eta_{\e,1}(w,s)=0&~\mbox{on}~\partial B(0,1),\\[1mm]
\eta_{\e,1}(w,s)=\frac{1}{2\pi}\left(\frac{|w|^2-1}{|w-z|^2}+\frac{\log|w|}{\log (r\e)}\right)&~\mbox{on}~\frac{\partial\O-x_0}\e.
\end{cases}
\end{equation}
Setting $z=x_0+\e w$, for any $w\in \frac{\partial\O-x_0}\e$,  we get that
\begin{equation*}
\begin{split}
 \frac{|w|^2-1}{|w-s|^2}+\frac{\log|w|}{\log (r\e)}=&\frac{|z-x_0|^2-\e^2}{|z-x_0-\e s|^2}+\frac{\log (r|z-x_0|)-\log (r\e)}{\log (r\e)}\\=&
\frac{\log (r|z-x_0|)}{\log (r\e)}+O(\e).
\end{split}
\end{equation*}
Then taking some appropriate $r>0$ (for example, taking $r>0$ small such that $r|z-x_0|<1$ for any $z\in \partial\Omega$), we have
\begin{equation}\label{6-23-16}
\begin{split}
\frac{|w|^2-1}{|w-s|^2}+\frac{\log|w|}{\log (r\e)}\geq 0.
\end{split}
\end{equation}
Hence by the maximum principle, \eqref{6-23-15} and \eqref{6-23-16}, it holds
$$\eta_{\e,1}(w,s)\geq 0, ~\mbox{for any}~w\in \frac{\Omega-x_0}{\varepsilon}\backslash B(0,1).$$
Then Hopf's lemma gives us that
\begin{equation*}
\frac{\partial \eta_{\e,1}(\tau,s)}{\partial \nu_\tau}
<0, ~\mbox{for any}~\tau\in \partial B(0,1)~\mbox{with}~\nu_\tau=-\frac{\tau}{|\tau|}.
\end{equation*}
If  $\displaystyle\lim_{\varepsilon\rightarrow 0}|w|=1$,  by the sign-preserving property, for small $\varepsilon$, we find
\begin{equation}\label{6-23-31}
\sum^2_{i=1}\frac{\partial\eta_{\e,1}(w,s)}{\partial w_i}w_i
=-|w|\cdot\frac{\partial\eta_{\e,1}(w,s)}{\partial \nu_w}
\geq 0,~\mbox{with}~\nu_w=-\frac{w}{|w|}.
\end{equation}
Hence using \eqref{6-23-31}, we can compute that
\begin{equation*}
\begin{split}
\Big\langle \nabla_w L_{\e}(w),w\Big\rangle
=& \int_{\partial B(0,1)}\left(\sum^2_{i=1}\frac{\partial\eta_{\e,1}(w,s)}{\partial w_i}w_i-\frac{1}{2\pi\log (r\e)}
\right)u_0(x_0+\e s)d\sigma(s)\\ \geq  &
 \frac{1}{|\log \e|}\big(u_0(x_0)+o(1)\big)\geq \frac{u_0(x_0)}{2|\log \e|},
\end{split}\end{equation*}
which completes the proofs of  \eqref{d11-23-05}.
\end{proof}

From above computations, the precise asymptotic behavior of $\nabla u_\e$ and
$\nabla^2 u_\e$ near $\partial B(x_0,\varepsilon)$ can be stated as follows.
\begin{prop}\label{alemma-2}Let $w=\frac{x-x_0}{\e}$ and $|x-x_0|\to 0$, we have following results:
 \vskip 0.2cm

 \textup{(1).} For any fixed $C_0>1$, if $|w|\geq C_0$, then it holds
\begin{equation}\label{s11-15-01}
\frac{\partial u_\e(x)}{\partial x_i}=
\frac{\partial u_0(x)}{\partial x_i}+ \frac{u_0(x_0) (x_i-x_{0,i})}{ |\log\e|\cdot |x-x_0|^2}
+o\Big(\frac{1}{ |\log\e|\cdot |x-x_0|}\Big),
\end{equation}
and
\begin{equation}\label{s11-15-01a}
\begin{split}
\frac{\partial^2 u_\e(x)}{\partial x_i\partial x_j}=&
\frac{\partial^2 u_0(x)}{\partial x_i\partial x_j}+ \frac{u_0(x_0) }{ |\log\e|\cdot |x-x_0|^2}\Big(\delta_{ij}-\frac{ (x_i-x_{0,i})(x_j-x_{0,j})}{ |x-x_0|^2} \Big) \\&
+o\Big( \frac{1}{ |\log\e|\cdot |x-x_0|^2}\Big).
\end{split}\end{equation}

 \textup{(2).}
If $\displaystyle\lim_{\e\to 0}|w|=1$, then it holds
\begin{equation}\label{s11-15-01b}
\begin{split}
\big|\nabla u_\e(x)\big|
 \geq \frac{u_0(x_0)}{8\e|\log \e|}.
\end{split}\end{equation}
\end{prop}
\begin{proof}
 \textup{(1).} First by \eqref{6-26-1}, we have
\begin{equation}\label{6-26-1d}
 \frac{\partial u_\varepsilon(x)}{\partial x_i}= \frac{\partial u_0(x)}{\partial x_i}+ \frac{1}{\e}\frac{\partial K_{\e}(w)}{\partial w_i}+  \frac{1}{\e}\frac{\partial L_{\e}(w)}{\partial w_i}.
\end{equation}
Then from  \eqref{aaas11-15-04-1}, \eqref{aast11-23-05} and \eqref{6-26-1d}, we find
\begin{equation*} 
\begin{split}
 \frac{\partial u_\varepsilon(x)}{\partial x_i}=& \frac{\partial u_0(x)}{\partial x_i}+
 O\Big(\frac{1}{|w|^2}\Big) +
\frac{u_0(x_0)w_i}{\e\cdot|\log\e|\cdot |w|^2}+o\left(\frac{1}{ \e\cdot |w|\cdot |\log\e|}\right)
\\=&
\frac{\partial u_0(x)}{\partial x_i}+ \frac{u_0(x_0) (x_i-x_{0,i})}{ |\log\e|\cdot |x-x_0|^2}
+o\Big(\frac{1}{ |\log\e|\cdot |x-x_0|}\Big)+O\Big(\frac{\e^2}{|x-x_0|^2}\Big)\\=&
\frac{\partial u_0(x)}{\partial x_i}+ \frac{u_0(x_0) (x_i-x_{0,i})}{ |\log\e|\cdot |x-x_0|^2}
+o\Big(\frac{1}{ |\log\e|\cdot |x-x_0|}\Big).
\end{split}
\end{equation*}
Similarly,  by \eqref{6-26-1}, we know
\begin{equation*} 
 \frac{\partial^2 u_\varepsilon(x)}{\partial x_i\partial x_j}= \frac{\partial^2 u_0(x)}{\partial x_i\partial x_j}+ \frac{1}{\e^2}\frac{\partial^2 K_{\e}(w)}{\partial w_i\partial w_j}+
 \frac{1}{\e^2}\frac{\partial^2 L_{\e}(w)}{\partial w_i\partial w_j},
\end{equation*}
which, together with \eqref{aaas11-15-04-1} and  \eqref{aamt11-23-05}, implies
 \eqref{s11-15-01a}.
\vskip 0.2cm

 \textup{(2).}
If $\displaystyle\lim_{\e\to 0}|w|=1$,  from \eqref{6-26-1}, \eqref{aaas11-15-04-1} and  \eqref{d11-23-05},
 we have
\begin{equation*}
\begin{split}
\big|\nabla u_\e(x)\big| \geq &\frac{1}{2}
\Big|\big\langle\nabla_x u_\varepsilon(x), w\big\rangle\Big|\\
\geq &  \frac{1}{2\e} \Big|\big\langle\nabla_w L_{\e}(w), w\big\rangle\Big|
-\frac{1}{2} \Big|\big\langle\nabla_x u_0(x), w\big\rangle\Big|
 - \frac{1}{2\e} \Big|\big\langle\nabla_w K_{\e}(w), w\big\rangle\Big| \geq \frac{u_0(x_0)}{8\e|\log \e|}.
 \end{split}
\end{equation*}
\end{proof}
\section{Proofs of Theorem \ref{th1.1}}\label{s9}

Firstly, we give the precise location of the maximum point $x_\e$ of $u_\e(x)$ on $\Omega_\e$.

\begin{prop}\label{Prop1-2a}
If $x_0\neq y_0$, then the maximum point $x_\e$ of $u_\e(x)$ on $\Omega_\e$
satisfies
$$x_\e \to y_0~~ \mbox{as}~~\e\to 0,$$ where $y_0$ is the   maximum point of  $u_0(x)$.
\end{prop}
\begin{proof}
First, for $x\in\Omega_\e$ with $|x-x_0|=o(1)$, it holds  $\log  |x-x_0|<0$, and then from \eqref{s11-15-0d1}, we find
\begin{equation}\label{y11-15-02}
u_\e(x)\leq
u_0(x)+o(1),~\mbox{in}~\big\{x\in\Omega_\e,~ ~|x-x_0|=o(1)\big\}.
\end{equation}
If $x_0\neq y_0$, then $u_0(x_0)<u_0(y_0)$, here we use the uniqueness of the critical point of $u_0(x)$(see \cite{ML71}). Hence by \eqref{y11-15-02},  we have
\begin{equation*}
u_\e(x)\leq \frac{u_0(x_0)+u_0(y_0)}{2}<u_0(y_0)\,\,~\mbox{in}~\big\{x\in\Omega_\e,~ ~|x-x_0|=o(1)\big\},
\end{equation*}
which gives us that $x_\e\notin \big\{x\in \Omega_\e, |x-x_0|=o(1)\big\}$.

Hence combining  Lemma \ref{alemma-1}, we know that
there exists a fixed small $r>0$ such that
$x_\e\in B(y_0,r)$ the maximum point $x_\e$ of $u_\e(x)$ will belong to $B(y_0,r)$
and $x_\e \to y_0$ as $\e\to 0$.

\end{proof}

\begin{prop}\label{Prop1-2}
If $x_0=y_0$, then the maximum point $x_\e$ of $u_\e(x)$ on $\Omega_\e$ can be written as
\begin{equation*} 
x_\varepsilon=x_0+
\sqrt{-\frac{u_0(x_0)+o(1)}\la}\frac{1}{\sqrt{|\log\e|}}v,
\end{equation*}
where $\lambda=\max\{\lambda_1,\lambda_2\}$,  $\lambda_1$ and $\lambda_2$ are the  eigenvalues of the matrix $D^2u_0(x_0)$ and $v$ an associated
eigenfunction with $|v|=1$.
\end{prop}
\begin{proof}
Since $\Omega$ is convex and $x_0=y_0$, then from \cite{ML71}, we know that $u_0(x)$ admits exact one critical point $x_0$. This means that for any $z_0\neq x_0$, there exists $r>0$ such that $u_0(x)$ has no critical points in $B(z_0,r)$. And from Lemma \ref{alemma-1}, we know that all critical points  of $u_\e(x)$ belong to $$D_\e:=\Big\{x\in \Omega_\e, |x-x_0|=o(1)\Big\}.$$
Next, for $x\in D_\e$, if $\frac{|x-x_0|}{\e}< C$,  from \eqref{s11-15-01} and \eqref{s11-15-01b}, we have
\begin{equation*}
\begin{split}
\big|\nabla u_\e(x)\big| \geq \frac{c_0}{\e|\log \e|}~~\,\mbox{for some}~~c_0>0,
 \end{split}
\end{equation*}
which implies that $\nabla u_\e(x)=0$ admits no solutions if $x\in D_\e$ and $\frac{|x-x_0|}{\e}< C$.

Finally, we analyze the critical points of $u_\e(x)$ on $x\in D_\e$ and $\frac{|x-x_0|}{\e}\to \infty$.
From \eqref{s11-15-01}, we can deduce that  
\begin{equation}\label{s11-16-01}
\begin{split}
0=&\frac{\partial u_\e(x_\e)}{\partial x_i}=
\frac{\partial u_0(x)}{\partial x_i}+ \frac{u_0(x_0) (x_i-x_{0,i})}{ |\log\e|\cdot |x-x_0|^2}
+o\Big(\frac{1}{ |\log\e|\cdot |x-x_0|}\Big)\\=&
\sum^2_{j=1}\left(\frac{\partial^2 u_0(x_0)}{\partial x_i\partial x_j}+o(1)\right)(x_{\e,j}-x_{0,j})+
\frac{x_{\e,i}-x_{0,i}}{|\log\e|\cdot |x_{\e}-x_0|^2}\Big(u_0(x_0)+o(1)\Big).
\end{split}\end{equation}
By \eqref{s11-16-01} we immediately get that $-\frac{1}{|x_{\e}-x_0|^2|\log \e|}\to\lambda$  as $\e\to0$.  Dividing \eqref{s11-16-01} by $|x_\e-x_0|$ and passing to the limit, we find that all critical points $x_\varepsilon$ of $u_\e$ can be written as
\begin{equation*} 
x_\varepsilon=x_0+
\sqrt{-\frac{u_0(x_0)+o(1)}\la}\frac{1}{\sqrt{|\log\e|}}v,
\end{equation*}
where $\lambda=\lambda_1$ or $\lambda_2$,  $\lambda_1$ and $\lambda_2$ are the  eigenvalues of the matrix $D^2u_0(x_0)$ and $v$ an associated
eigenfunction with $|v|=1$.

Now we devote to prove that 
the maximum point $x_\e$ of $u_\e(x)$ on $\Omega_\e$ can be written as
\begin{equation*} 
x_\varepsilon=x_0+
\sqrt{-\frac{u_0(x_0)+o(1)}\la}\frac{1}{\sqrt{|\log\e|}}v, ~\mbox{with}~\lambda=\max\{\lambda_1,\lambda_2\}.
\end{equation*}
And if $\lambda_1=\lambda_2$, then above result holds automatically. 
Now let 
 \begin{equation}\label{ll1}
 x_\varepsilon=x_0+
\sqrt{-\frac{u_0(x_0)+o(1)}{\la_1}}\frac{1}{\sqrt{|\log\e|}}v_1\,\,~\mbox{with}~D^2u_0(x_0)v_1=\lambda_1v_1.
\end{equation}
Then from \eqref{s11-15-01a}, we know
 \begin{equation}\label{ll2}
\begin{split}
\frac{\partial^2 u_\e(x_\e)}{\partial x_i\partial x_j}=&
\frac{\partial^2 u_0(x_0)}{\partial x_i\partial x_j}
-\lambda_1 \Big(\delta_{ij}-v_{1i}v_{1j} \Big)
+o\big(1\big).
\end{split}\end{equation}
Next we take $v_2$ satisfying $D^2u_0(x_0)v_2=\lambda_2v_2$  with $|v_2|=1$ and $v_1\bot v_2$.
And then denoting
 $\textbf{P}=\left(
 \begin{array}{cc}
 v_{11} & v_{21} \\
  v_{12} & v_{22} \\
  \end{array}
  \right)$, we have
\begin{equation}\label{ll3}
\begin{split}
\textbf{P}^{-1}D^2u_\e(x_\e)\textbf{P}
=& \textbf{P}^{-1}D^2u_0(x_0)\textbf{P}-\Big(\lambda_1+o(1)\Big)\textbf{E}
+\lambda_1\textbf{P}^{-1}v_1v_1^T\textbf{P}.
\end{split}\end{equation}
where $\textbf{E}$ is the unit matrix. Also we compute that
\begin{equation}\label{ll4}
\begin{split}
\textbf{P}^{-1}v_1v_1^T\textbf{P}=&
\left(
 \begin{array}{cc}
 v_{11} & v_{12} \\
  v_{21} & v_{22} \\
  \end{array}
  \right)   \left(
              \begin{array}{c}
                v_{11} \\
                v_{12}\\
              \end{array}
            \right)
\left(
  \begin{array}{cc}
     v_{11} & v_{12} \\
  \end{array}
\right)\left(
 \begin{array}{cc}
 v_{11} & v_{21} \\
  v_{12} & v_{22} \\
  \end{array}
  \right) \\=&   \left(
              \begin{array}{c}
                1 \\
                0\\
              \end{array}
            \right)
\left(
  \begin{array}{cc}
     1 & 0\\
  \end{array}
\right) =  \left(
 \begin{array}{cc}
 1 & 0 \\
0 & 0 \\
  \end{array}
  \right).
  \end{split}
\end{equation}
Hence from \eqref{ll1}, \eqref{ll2}, \eqref{ll3} and \eqref{ll4}, it follows
\begin{equation}\label{ll5}
\begin{split}
\textbf{P}^{-1}D^2u_\e(x_\e)\textbf{P}
=&   \left(
 \begin{array}{cc}
\lambda_1+o(1) & 0 \\
0 & \lambda_2-\lambda_1+o(1) \\
  \end{array}
  \right).
\end{split}\end{equation}
This gives us that 
\begin{equation*} 
\begin{split}
\mbox{det}~D^2u_\e(x_\e)=\Big(\lambda_1+o(1)\Big)\Big(\lambda_2-\lambda_1+o(1)\Big).
\end{split}\end{equation*}

If $\lambda_1<\lambda_2$, then $\mbox{det}~D^2u_\e(x_\e)<0$ and $x_\e$ is a saddle point of $u_\e(x)$. Hence in this case, $x_\e$ is not a maximum point of $u_\e(x)$. If $\lambda_1>\lambda_2$, then two eigenvalues of $D^2u_\e(x_\e)$ are negative and $x_\e$ is a maximum point of $u_\e(x)$. These complete the proof of Proposition \ref{Prop1-2}.
\end{proof}

\vskip 0.2cm

Now we are ready to prove Theorem \ref{th1.1}.

\vskip 0.2cm

\begin{proof}[\underline{{\textbf{Proof of Theorem \ref{th1.1}}}}] We divide into following two cases.
 \vskip 0.2cm

\noindent\textbf{Case 1: $x_0\neq y_0$.} First from Proposition \ref{Prop1-2a}, we know that
the maximum point $x_\e$ of $u_\e(x)$ on $\Omega_\e$
satisfies
$$x_\e \to y_0~~ \mbox{as}~~\e\to 0.$$
Also
we recall that all  eigenvalues of $D^2u_0(y_0)$ are negative. Hence by continuity,
\eqref{5-10-1} and Lemma \ref{alemma-1}, we find that
\begin{equation}\label{ll7}
\displaystyle\lim_{\e\to 0}\lambda_{\max}\big(D^2u_\e(x_\e)\big)=
\lambda_{\max}\big(D^2u_0(y_0)\big)
=\max\Big\{\lambda_1,\lambda_2\Big\}<0.
\end{equation}
 \vskip 0.1cm

\noindent\textbf{Case 2: $x_0=y_0$.}
Without loss of generality, we suppose that $\lambda_2\leq \lambda_1<0$  then from Proposition \ref{Prop1-2}, we know that
the maximum point of $u_\e(x)$ satisfies
 \begin{equation*}
 x_\varepsilon=x_0+
\sqrt{-\frac{u_0(x_0)+o(1)}{\la_1}}\frac{1}{\sqrt{|\log\e|}}v_1\,\,~\mbox{with}~D^2u_0(x_0)v_1=\lambda_1v_1.
\end{equation*} 
And then \eqref{ll5} gives us that
\begin{equation*}
 \lim_{\e\to 0}\lambda_{\max}\big(D^2u_\e(x_\e)\big)=
 \max\Big\{\lambda_1,\lambda_2-\lambda_1\Big\}
\begin{cases}
<0,~\mbox{for}~\lambda_2<\lambda_1,\\[1mm]
 =0, ~\mbox{for}~\lambda_2=\lambda_1.
 \end{cases}
\end{equation*}
Hence for general $\lambda_1,\lambda_2$, it holds
\begin{equation*}
 \lim_{\e\to 0}\lambda_{\max}\big(D^2u_\e(x_\e)\big)=\max\Big\{\lambda_1,\lambda_2,-|\lambda_2-\lambda_1|\Big\}
\begin{cases}
<0,~\mbox{for}~\lambda_2\neq \lambda_1,\\[1mm]
 =0, ~\mbox{for}~\lambda_2=\lambda_1,
 \end{cases}
\end{equation*}
which, together with \eqref{ll7},  completes the proof of Theorem \ref{th1.1}.
\end{proof}

\vskip 0.2cm

\noindent\textbf{Acknowledgments} ~Part of this work was done while Peng Luo  was visiting the Mathematics Department
of the University of Rome ``La Sapienza" whose members he would like to thank for their warm hospitality. Hua Chen was supported by NSFC grants (No. 11631011,11626251). Peng Luo was supported by NSFC grants (No.11701204,11831009).

\end{document}